\documentclass[11pt]{amsart}
\usepackage[active]{srcltx}

\usepackage{amsmath, amsfonts,amsthm,times,graphics,color}
\newcommand{\vertiii}[1]{{\left\vert\kern-0.25ex\left\vert\kern-0.25ex\left\vert #1
    \right\vert\kern-0.25ex\right\vert\kern-0.25ex\right\vert}}
 \makeatletter
\renewcommand*\subjclass[2][2000]{%
  \def\@subjclass{#2}%
  \@ifundefined{subjclassname@#1}{%
    \ClassWarning{\@classname}{Unknown edition (#1) of Mathematics
      Subject Classification; using '1991'.}%
  }{%
    \@xp\let\@xp\subjclassname\csname subjclassname@#1\endcsname
  }%
}
 \makeatother
\usepackage{enumerate,amssymb,  mathrsfs,yhmath}

\newtheorem{theorem}{Theorem}[section]
\newtheorem{lemma}[theorem]{Lemma}
\newtheorem*{lemma*}{Lemma}
\newtheorem{proposition}[theorem]{Proposition}
\newtheorem{corollary}[theorem]{Corollary}
\newtheorem{definition}[theorem]{Definition}

\def\1ton{1,2,\ldots,n}

\usepackage{amssymb}
\usepackage{amsthm}
\usepackage{mathrsfs, amsfonts, amsmath}
\usepackage{graphicx}

\newcommand{\R}{\mathbb{R}}

\newcommand{\B}{\mathbb{B}}

\newtheorem{remark}[theorem]{Remark}

\numberwithin{equation}{section}

\newcommand{\abs}[1]{\lvert#1\rvert}

\renewcommand{\imath}{i} 

\def\XXint#1#2#3{{\setbox0=\hbox{$#1{#2#3}{\int}$}
\vcenter{\hbox{$#2#3$}}\kern-.5\wd0}}

\def\ge{\geqslant}
\begin{document}

\title[Contraction property of  certain classes of subharmonic functions in $\mathbb{R}^n$]{Contraction property of certain classes of log-$\mathcal{M}-$subharmonic functions  in the unit ball}


\author{David Kalaj}
\address{University of Montenegro, Faculty of Natural Sciences and
Mathematics, Cetinjski put b.b. 81000 Podgorica, Montenegro}
\email{davidk@ucg.ac.me}

\begin{abstract}
 We prove a contraction property of certain classes of smooth functions, whose absolute values of elements are  log-hyperharmonic functions in the unit ball, thus extending the results of Kulikov to higher-dimensional space (GAFA (2022)). Moreover, by applying those results we get some new results for harmonic mappings in the complex plane.
\end{abstract}
\maketitle
\tableofcontents
\sloppy

\maketitle
\section{Introduction}
In this paper  $\mathbb{B} = \{ x\in \mathbb{R}^n : |x| < 1\}$ is the unit ball. Here and in the sequel for $x=(x_1,\dots, x_n)$, $|x|:=\sqrt{\sum_{k=1}^n x_k^2}$.

Assume that $\mathcal{M}$ is the group of M\"obius transformations of the unit ball onto itself (see e.g. \cite{alf}).
We introduce the M\"{o}bius invariant hyperbolic measure on the unit ball. For $x\in \mathbb{B}$ we define it as
$$d\tau(x) = \frac{2^n}{(1-|x|^2)^n} \frac{dV(x)}{{\omega_n}},$$ where $\omega_n=V(B)$ is the volume of the unit ball.

A mapping $u\in C^{2}(\mathbb{B}^{n}, \mathbb{C})$ or more generally $u\in C^{2}(\mathbb{B}^{n}, \mathbb{R}^k)$ is said to be {\it hyperbolic harmonic} or $\mathcal{M}-$\emph{harmonic}  if $u$ satisfies the hyperbolic Laplace equation
\begin{equation}\label{hyphar}
\Delta_{h}u(x)=(1-|x|^2)^2\Delta u(x)+2(n-2)(1-|x|^2)\sum_{i=1}^{n} x_{i} \frac{\partial u}{\partial x_{i}}(x)=0,
\end{equation}
 where
 $\Delta$ denotes the usual Laplacian in $\mathbb{R}^{n}$. We call $\Delta_{h}$ the {\it hyperbolic Laplace operator}.
See Rudin \cite{rudin} and Stoll \cite{stoll}.

The Poisson kernel for $\Delta_h$ is defined by $$P_h(x,\zeta)=\frac{(1-|x|^2)^{n-1}}{|x-\zeta|^{2n-2}}, \ \ (x,\zeta)\in \B\times \mathbb{S}.$$

Then for fixed $\zeta$, $x\to P_h(x,\zeta)$ is $\mathcal{M}-$harmonic function and for a mapping $f\in \mathcal{L}(\mathbb{S})$, the function $$u(x) = P_h[f](x):=\int_{\mathbb{S}}P_h(x,\zeta) f(\zeta) d\sigma(\zeta)$$ is the Poisson extension of $f$ and it is $\mathcal{M}-$harmonic in $\B$.

 We say that a real function $u$ is $\mathcal{M}-$subharmonic if $\Delta_h u(x) \ge 0$. The definition can be extended to the case of upper semicontinuous functions, by using the so-called invariant mean value property \cite{stoll}.

Now we define the Hardy type space.
\begin{enumerate}
\item
As in \cite[Definition~7.0.1]{stoll}, for $0<p\le \infty$, we denote by $\mathcal{S}^p$ the {\it Hardy-type} space of non-negative $\mathcal{M}-$subharmonic functions $f$ on $\B$ such that
\begin{equation}\label{meanconv}
\|f\|_p^p=\sup_{0<r<1}\int_{\mathbb{S}}|f(r\zeta)|^pd\sigma(\zeta)<\infty.
\end{equation}
When $p=\infty$ we put $\|f\|_\infty=\sup_{x\in\B} f(x)$.
\item  For $0 < p < \infty$ we say that a Borel function $f:\mathbb{B}\to \mathbb{C}$ belongs to the Hardy space $h^p$ if $|f|\in \mathcal{S}^p$. Then we define
$\|f\|_p := \| |f|\|_p$. When $p=\infty$ we put $\|f\|_{\infty}=\sup_{x\in\B} |f(x)|$.
\end{enumerate}

 Further \begin{equation}\label{fol}(\Delta_h u)(m(x)) =\Delta_h (u\circ m)(x),\end{equation} for every M\"obius transformation $m\in\mathcal{M}$
 of the unit ball onto itself.

 For $n=2$ the $\mathcal{M}-$harmonic and $\mathcal{M}-$subharmonic functions are just harmonic and subharmonic functions.

If $f$ is $\mathcal{M}-$subharmonic, then we have  the following Riesz decomposition theorem of Stoll \cite[Theorem~9.1.3]{stoll}:  $$f(x) = F_f(x)- \int_{\B} G_h(x,y) d\mu_f(y),$$ provided that $f\in \mathcal{S}^1$, where  $F_f(x)$ is the least $\mathcal{M}-$harmonic majorant of $f$ and $\mu_f$ is the $\mathcal{M}-$ Riesz measure of $f$, and $G_h(x,y)$ is the Green function of $\Delta_h$. If $f\in \mathcal{S}^p$, where $p>1$, then $g(x)=F_f(x)=P_h[\hat f](x)$, where $\hat f$ is the boundary function of $f$ (\cite[Theorem~7.1.1]{stoll}).

From the formula \eqref{fol}, by putting $u=\mathrm{Id}$, and $m\in\mathcal{M}$, we get
$$\Delta_h m=2 (n-2)(1-|m|^2)m.$$ So M\"obius transformations are (considered as vectorial functions) hyperbolic harmonic only in the case $n=2$.

 By putting $u(x) = g(|x|)$ and inserting in \eqref{hyphar} we arrive to the equation $$\Delta_h u= \left(1-r^2\right) \left(\frac{\left(2 (-2+n) r^2+(-1+n) \left(1-r^2\right)\right) g'(r)}{r}+\left(1-r^2\right) g''(r)\right), $$ where $r=|x|$.
The hypergeometric function $F$, which we use in this paper is defined by $$F\left[\begin{array}{c}
                                           a,b,c \\
                                           u,v
                                         \end{array};t\right]:=\sum_{n=0}^{\infty}\frac{(a)_{n}(b)_{n}(c)_n}{n!(u)_n (v)_n}t^{n},\enspace \mbox{for}\enspace|t|<1,$$
   and by the continuation elsewhere. Here $(a)_n$ denotes { the} shifted factorial, i.e., $(a)_{n}=a(a+1)...(a+n-1)$ {
with   any real number $a$}.

Then one of solutions to the equation $\Delta_h \log v= -4(n-1)^2$, is given by (see Proposition~\ref{pbelo} below)
 \begin{equation}\label{Phi}\Phi_n(r)=\exp\left\{\frac{(n-1) (2-n) r^2}{ n} F\left[\begin{array}{c}
                                           1,1,2-\frac{n}{2} \\
                                           2,1+\frac{n}{2}
                                         \end{array}; r^2\right]\right\} \left(1-r^2\right)^{n-1}.\end{equation}
                                         We will also sometimes write $\Phi_n(x)$ instead of $\Phi_n(|x|)$.
                                        If $n=2$, then $\Phi_n(|x|) = 1-|x|^2$ and this coincides with the case treated in \cite{kulik} by Kulikov.

\begin{definition}
For $0 < p < \infty$ and $\alpha > 1$ we say that a complex smooth function $f$  in $\mathbb{B}$ belongs to the $\mathcal{M}-$Bergman space ${B}^p_\alpha$ if
$$\|f\|_{\alpha,p}^p = c(\alpha)\int_\mathbb{B} |f(x)|^p \Phi_n^\alpha(|x|) d\tau(x) < \infty,$$ where
\begin{equation}\label{calpha}\frac{1}{c(\alpha)}= \int_{\B} \Phi^\alpha_n(x)(1-|x|^2)^{-n}\frac{dV(x)}{{\omega_n}},\end{equation} and $\Phi_n$ is a function defined in \eqref{Phi} above. \end{definition}
Observe that in view of \eqref{esfi}, $$\frac{1}{c(\alpha)}\ge E_n^\alpha \int_0^1 r^{n}(1-r^2)^{\alpha(n-1)-n} dr.$$ This implies that \begin{equation}\label{caz}
\lim_{\alpha\to 1+}c(\alpha)=0.
\end{equation}

\subsection{Admissible monoid}
We define $\mathfrak{E}$ to be the set of complex continuous functions $g$ in $\mathbb{B}$ and  real analytic  in $g^{-1}(\mathbb{C}\setminus\{0\})$   such that
$f:=\log |g|$ is $\mathcal{M}-$subharmonic. If $g(a)=0$, then we put $f(a)=-\infty$. Let us remind  that in the definition of subharmonicity, based on the mean value property, continuity of the function is not assumed a priori, but only upper continuity and the value $-\infty$ for the function is a legitimate value.

 Let  $\mathfrak{E}_+=\{f\in \mathfrak{E}: f\ge 0\}$. Observe that $\mathfrak{E}$ is a monoid where the operation is simply the multiplication of two functions.  Observe that $1=e^0$, so $1\in \mathfrak{E}$. This monoid contains the Abelian group $\mathcal{G}=\{e^f: \Delta_hf=0\}$.
Then for $a,b\in\mathfrak{E}_+$, $c, d\in \mathfrak{E}$,  $p\ge 0$ and $\alpha, \beta>0$ we have
\begin{enumerate}
\item $a\cdot b \in\mathfrak{E}_+$, $c\cdot d\in \mathfrak{E}$,
\item $a^p\in \mathfrak{E}_+$, $c^p\in \mathfrak{E}$,
\item $\alpha a +\beta b \in \mathfrak{E}_+$.
\end{enumerate}

In other words, $\mathfrak{E}_+$ is a convex cone and at the same time a monoid.

Previous statements follow from the straightforward calculation of the invariant Laplacian.
First of all, we have
$$\alpha e^f = e^{\log \alpha +f}.$$ So if $a\in \mathfrak{E}_+$ and $\alpha>0$, then so is $\alpha a$.
Further if $a=e^f$ and $b=e^g$, then $c=a+b= e^{\log (e^f+e^g)}$. Thus
$$\Delta_h\log c=\Delta_h  \log (e^f+e^g)=\frac{e^f\Delta_h f +e^g\Delta_h g}{e^f+e^g}+\frac{e^{f+g}}{e^f+e^g}(1-r^2)|\nabla (f-g)|^2,$$
where
$$A=\Delta_h g=f,\  \ B=\Delta_h f=g.$$

So if $A, B\ge 0$, then $\Delta_h  \log (e^f+e^g)\ge 0$. We also refer to the paper \cite{jfa} for some similar properties of log-subharmonic functions.

Let us collect some additional features of $\mathcal{M}-$harmonic and $\mathcal{M}-$subharmonic functions.
\begin{enumerate}
\item
If $\log f$ is $\mathcal{M}-$subharmonic, then $f=e^{\log f}$ is $\mathcal{M}$-subharmonic.
\item
If $\log f$ is $\mathcal{M}-$subharmonic, then $p\log f$ is $\mathcal{M}$-subharmonic and so $f^p=e^{p\log f}$ is $\mathcal{M}$-subharmonic.
\item
If $f$ is $\mathcal{M}-$harmonic, then $|f|$ is $\mathcal{M}-$subharmonic.
\item
If $f$ is $\mathcal{M}-$harmonic, then $f\circ m$ is $\mathcal{M}-$harmonic, for every M\"obius transformation $m\in\mathcal{M}$ of the unit ball onto itself.
\item
If $f$ is $\mathcal{M}-$harmonic, then $|f|^p$ is $\mathcal{M}-$subharmonic for $p\ge 1$.

\end{enumerate}
For the above facts, we refer to the monograph by Stoll \cite{stoll}.
Let us illustrate the proof of one of the properties: If $u$ is $\mathcal{M}-$subharmonic, then $e^u$ is $\mathcal{M}-$subharmonic.
Indeed
 \begin{equation}\label{posexp}\begin{split}\Delta_h e^u &=  (1-|x|^2)^2 |\nabla u|^2 e^u \\&+ e^u((1-|x|^2)^2\Delta u(x)+2(n-2)(1-|x|^2)\sum_{i=1}^{n} x_{i} \frac{\partial u}{\partial x_{i}}(x))\\&=(1-|x|^2)^2 |\nabla u|^2 e^u +e^u \Delta_h u\ge 0.\end{split}\end{equation}
\begin{remark}
Unfortunately, the class of holomorphic functions in $\mathbb{B}\subset \mathbb{C}^n\cong \mathbb{R}^{2n}$ for $n\ge 2$ is not contained in the monoid $\mathfrak{E}$. To see this, let $n=2$ and take the holomorphic mapping $f(w,z)=g(z)$ which depends only on $z$. Here $g$ is a holomorphic non-vanishing function. Then straightforward computations give that $$\Delta_h \log |f(w,z)|=8 (1-|w|^2-|z|^2)\Re\left(\frac{z g'(z)}{g(z)}\right), \ |z|<1.$$ Hence the function $f$ is $\mathcal{M}-$subharmonic if and only if $g$ is a starlike function, i.e. if $\Re\left(\frac{z g'(z)}{g(z)}\right)\ge 0$ for $|z|<1$.
\end{remark}

\begin{definition}\label{monomono}
 For $0 < p < \infty$ and $\alpha > 1$ we define
 \begin{enumerate}
\item the $\mathcal{M}-$ Hardy  monoid  $\mathbf{h}^p$ consisting of functions $f$ in $h^p\cap\mathfrak{E}$: satisfying the additional condition that there is a sequence of bounded $\mathcal{M}-$log-subharmonic functions $f_m$ converging in $h^p$ norm and poinwisely to $f$.
 \item the $\mathcal{M}-$Bergman monoid $\mathbf{B}^p_\alpha$ consisting of functions $f$ in ${B}^p_\alpha\cap \mathfrak{E}$. 
\end{enumerate}
We also refer to \cite{vuko} for the corresponding Bergman space of holomorphic mappings in the space which is different from our space.
\end{definition}
\begin{remark}
{Note that the additional condition in the previous definition is redundant for smooth log-subharmonic functions for $n=2$. In this case the dilatations $f_m(z)=f(\rho z)$ are log-subharmonic functions, $\rho=\rho_m=\frac{m}{m+1}$ and converge in ${h}^p$ norm and pointwisely  to $f$. See e.g. \cite[Proposition~1.3]{hakan}.} We also believe it is redundant in higher-dimensional case.
\end{remark}

Observe that for $f(x) \equiv 1$ we have $\|f\|_{p} = \|f\|_{\alpha,q} = 1$ for all  $p, q>0$ and $\alpha>1$.

An important property of these spaces is that point evaluations are continuous functionals. For this fact see Proposition~\ref{propo1}.

One of the interesting questions about those spaces is which space is the subset of the other space. We consider the case when the quotient  $\frac{p}{\alpha}$ is held constant, in which case we have
$$\mathbf{B}^p_\alpha \subset \mathbf{B}^q_\beta,\quad \frac{p}{\alpha} = \frac{q}{\beta} = r,\quad p < q$$
and $\mathbf{h}^r$ is contained in all these spaces. This follows from our results.

Recently for $n=2$, it was asked whether these embeddings are contractions, that is whether the norm $  \|f \|_{\mathbf{B}^{r\alpha}_\alpha}$ is decreasing in $\alpha$. In the case of Bergman spaces, this question was asked by Lieb and Solovej \cite{9}. They proved that such contractivity implies their Wehrl-type entropy conjecture for the $SU(1,1)$ group. In the case of contractions from the Hardy spaces to the Bergman spaces, it was asked by Pavlovi\'c in \cite{13} and by Brevig, Ortega-Cerd\`a, Seip, and Zhao \cite{7} concerning the estimates for analytic functions. In a recent paper \cite{kulik}, Kulikov confirmed these conjectures, and  he proved more general results where the function $t^r$ is replaced with a general convex or monotone function, respectively.

\subsection{Statement of main results}
In this paper, we extend all those results to the higher-dimensional space by proving the following theorems.

\begin{theorem}\label{Hardy}
Let $p>0$. $G:[0,\infty)\to \R$ be an increasing function. Then the maximum value of
\begin{equation}\label{hardyineq}
\int_\mathbb{B} G(|f(x)|^p\Phi_n(x))d\tau(x)
\end{equation}
is attained for $f(x) \equiv 1$, subject to the condition that $f\in \mathbf{h}^p $ and $ \|f \|_{p}=1$.
\end{theorem}
\begin{theorem}\label{Bergman}
Let $p>0$ and $\alpha>0$. Let $G:[0,\infty)\to \R$ be a convex function. Then the maximum value of
\begin{equation}\label{bergmanineq}
\int_\mathbb{B} G(|f(x)|^p\Phi_n(x))^\alpha)d\tau(x)
\end{equation}
is attained for $f(x) \equiv 1$, subject to the condition that $f\in \mathbf{B}^p_\alpha$ and $ \|f \|_{\alpha,p}=1$.
\end{theorem}
We will prove Theorem~\ref{Bergman} in Section~\ref{sec4}. We will first prove Theorem~\ref{Hardy} for bounded log-subharmonic functions and use such a statement for proving Theorem~\ref{shtune} below. More precisely, Theorem~\ref{monotone} and Proposition~\ref{propozicija} \eqref{alphap} below imply Theorem~\ref{kalajda} for a bounded function $f$. Further such a statement implies  Theorem~\ref{Hardy} for a bounded function $f$. Furthermore, Theorem~\ref{Bergman} and Theorem~\ref{Hardy} for bounded  functions imply Theorem~\ref{shtune}.  Then by using Theorem~\ref{shtune} we prove  Proposition~\ref{propozicija1}, which completes the missing part of the proof of Theorem~\ref{kalajda}  and this implies Theorem~\ref{Hardy} in full generality.
To prove Theorem~\ref{shtune} we apply  the convex and increasing function $G(t) = t^s, s > 1$. Note that we get that all the embeddings above between Hardy and Bergman monoids are contractions.
\begin{theorem}\label{shtune}
For all $0 < p < q < \infty$ and $1 < \alpha < \beta < \infty$ with $\frac{p}{\alpha}=\frac{q}{\beta} = r$ and for all $f\in \mathbf{h}^r$,  we have
\begin{equation}\label{two2}\|f\|_{\beta, q} \le \|f\|_{\alpha,p} \le \|f\|_{h^r}\end{equation}
with equality for $f(z) \equiv c$, where  $c\in \mathbb{C}$, or for $f$ belonging to the extremal set below.
\end{theorem}
\begin{proof}

First of all, by the definition, there is a sequence of bounded $\mathcal{M}-$log-subharmonic functions $f_m$ converging  to $f$. Then Lemma~\ref{hardy2023} (for bounded function $f_m$) imply  the inequality \begin{equation}\label{two23}\|f_m\|_{\alpha,p} \le \|f_m\|_{h^r}.\end{equation} By letting $m\to\infty$, using the poinwise convergence in the left-hand side and the ${h^r}$ convergence of sequence in the right-hand side of inequality \eqref{two23}, in view of Fatou's lemma we get \begin{equation}\label{two232}\|f\|_{\alpha,p} \le \|f\|_{h^r}.\end{equation}
If  $f\in \mathbf{h}^r$, then $f\in \mathbf{B}^p_\alpha$ and Theorem~\ref{Bergman} implies that \begin{equation}\label{1two21}\|f\|_{\beta, q} \le \|f\|_{\alpha,p}.\end{equation}
\end{proof}

\subsection{Extremal set}
It is important to mention that the M\"{o}bius group acts not only on the measure $\tau$ but on the spaces $\mathbf{B}^p_\alpha$ as well. More precisely, given a function $f\in \mathbf{B}^p_\alpha$ and $m\in \mathcal{M}$, the function
\begin{equation}\label{prety}g(x) = f\left(m(x)\right)\frac{\Phi_n^{\alpha/p}(|m(x)|)}{\Phi_n^{\alpha/p}(|x|)}\end{equation}
also belongs to the space $\mathbf{B}^p_\alpha$ and moreover it has the same norm as $f$ and the same distribution of the function $|f(x)|^p\Phi_n^\alpha(x)$ with respect to the measure $\tau$. We need to check that $\Delta_h \log g\ge 0$, if $\Delta_h \log f\ge 0$, and this follows from the formula \eqref{fol} and straightforward calculations:
\[\begin{split}\Delta_h\log g(x)&=\Delta_h \log(f\left(m(x)\right)+\Delta_h \log {\Phi_n^{\frac{\alpha}{p}}(|m(x)|)})\\&+\Delta_h \log {\Phi_n^{\frac{\alpha}{p}}(|x|)}
\\&=\Delta_h \log(f\left(y\right))_{y=m(x)}+\Delta_h{\log\Phi_n^{\frac{\alpha}{p}}(|y|)}|_{y=m(x)}-\Delta_h{\log\Phi_n^{\frac{\alpha}{p}}(|y|)}_{y=x}\\&\ge 0 + (4(n-1)^2-4(n-1)^2)\frac{\alpha}{p}= 0.\end{split}\]

 To prove the second statement, we only need to point out the well-known formula for the Jacobian of M\"obius transformations of the unit ball onto itself $$J_m(x)=\frac{(1-|m(x)|^2)^n}{(1-|x|^2)^n}.$$ See e.g. \cite[p.~vii]{stoll}.

In particular, when $f(x) \equiv 1$  we get $g(x) = \frac{\Phi_n^{\alpha/p}(|m(x)|)}{\Phi_n^{\alpha/p}(|x|)}$ and the function $g$ give us the maximal value in \eqref{hardyineq} and \eqref{bergmanineq} for every $m\in\mathcal{M}$.

We believe that our results also can be formulated for the Hardy and Bergman spaces in the upper half-space, by using a conformal mapping from the unit ball or by directly translating our methods.
\subsection{Structure of the paper}
  The paper contains 5 more sections. In Section~\ref{sec1} we prove a general monotonicity theorem for the hyperbolic measure of the  superlevel sets of log-$\mathcal{M}$-subharmonic functions, which is an adaptation of the beautiful method from \cite{3, kulik}. Then, in Sections~\ref{sec3} and \ref{sec4} we deduce from it Theorems  \ref{Hardy}  and \ref{Bergman}, respectively. Notice that the proof of Theorem~\ref{Bergman} is even simpler than the proof of the analogous theorem in \cite{kulik} for the planar case. Finally, in Section~\ref{sec5} we briefly discuss an application of Corollary~\ref{shtune} to  coefficient estimates for harmonic functions and some important classes of log-subharmonic functions. In the Appendix below it is given a solution to the hyperbolic harmonic equation in the unit ball and there are proven some propositions that deal with Hardy and weight-Bergman spaces of $\mathcal{M}-$subharmonic functions used in this paper.
\section{The proof of the main result}\label{sec1}
We begin with
\subsection{Isoperimetric inequality for the hyperbolic ball and the function $\Upsilon$}

For a Borel set $E\subset \mathbb{B}$ we recall the definition of the hyperbolic volume  $$|E|_h=V(E) =\int_{E}\left(\frac{2}{1-|x|^2}\right)^n dx.$$ Moreover the hyperbolic perimeter is defined by
$$|\partial E|_h=P(E)=\int_{\partial E}\left(\frac{2}{1-|x|^2}\right)^{n-1} d\mathcal{H}^{n-1}(x).$$
Assume that $\B_s$ is the ball centered at the origin with the radius $\tanh\frac{s}{2}$.

The isoperimetric property of hyperbolic ball was established by Schmidt \cite{14} see also \cite{cvpde, 15}. He proved that for every Borel set $E\subset \mathbb{B}$ of finite perimeter $P(E)$, such that  $V(E)=V(\mathbb{B}_s)$ and $s>0$ we have
\begin{equation}\label{isophyper}
P_s\le P(E),
\end{equation}
where $P_s$ is the perimeter of $\B_s$ defined by
$$P_s=n \omega_n\sinh^{n-1}(s)= P(\B(0, \tanh s/2)).$$ The volume of $\B_s$ is given by
$$V_s=v(s):=n \omega_n \int_0^s \sinh^{n-1}(t)dt=V(\B(0, \tanh s/2)).$$ Here
$\omega_n= \frac{\pi^{n/2}}{\Gamma[n/2+1]}$ is the Euclidean volume of the unit ball. Since $s\to v=V_s$ is increasing, it has an inverse function $S(v)=s$.
Then define the function $\Upsilon$ by  \begin{equation}\label{newper}\Upsilon(v)=\frac{v}{P^2_{S(v)}} \end{equation} and thus by \eqref{isophyper}
\begin{equation}\label{newper1}P(E)^2/V(E) \ge 1/\Upsilon(V(E)),\end{equation} with an equality in \eqref{newper1} if and only if $E$ is a ball.
In the following remark, we give a connection with the isoperimetric inequality in the Euclidean space and it is given a specific inequality for $n=2$.

\begin{remark}
\emph{For $s>0$ we have
$$\frac{V_s}{P_s^{n/(n-1)}}=\frac{n \omega_n \int_0^s \sinh^{n-1}(t)dt}{n^{n/(n-1)} \omega_n^{n/(n-1)}\sinh^{n}(s)}
=\frac{\psi(s)}{\phi(s)}=\frac{\psi'(t)}{\phi'(t)}=c_0\frac{1}{\cosh t}\le c_0$$
where $$c_0 =n^{-\frac{n}{-1+n}} \omega_n^{\frac{1}{1-n}} .$$}
Now
$$P^2\ge V/\Upsilon (V)\ge c V^{2(n-1)/n} $$ where
$$c=n^2 \omega _n^{2/n}.$$ It can be proved that
\begin{equation}\label{newes}P^n- c^{n/2} V^{n-1}=P^n-n^n \omega_n V^{n-1}\ge (n-1)^n V^n.\end{equation}
If $n=2$ then the estimate \eqref{newes} is equivalent to the estimate \eqref{newper1} (for the hyperbolic plane of negative Gaussian curvature $-1$ see \cite{16}). In this case $\Upsilon(V)=\frac{1}{4\pi + V}$. It seems unlikely that we can give an explicit expression for the function $\Upsilon$ for a higher-dimensional case, but we don't need it in the proofs of our results.
\end{remark}

Let $f$ be a real analytic complex valued function such that  $v=|f|$ is log-$\mathcal{M}$-log-subharmonic function in $\mathbb{B}$  and such that $u(x) = v(x)^a \Phi_n^\alpha(x) $ is bounded and goes to $0$
uniformly as $|x|\to 1$. Then the superlevel sets $A_t = \{x: u(x) > t\}$ for $t > 0$ are
compactly embedded into $\mathbb{B}$ and thus have finite hyperbolic measure $\mu(t) = \tau(A_t)$.

In this section, we prove the following theorem which says that a certain function related to this measure is decreasing.
\begin{theorem}\label{monotone}
Let $\alpha\ge 1$ and $a\ge 0$ and assume that $f$ is a real analytic complex valued function such that $v=|f|:\mathbb{B} \to [0,+\infty)$ is a log-$\mathcal{M}$-subharmonic function. Assume further that the function $u(x) = |f(x)|^a\Phi^\alpha_n(x)$ is bounded and $u(x)$ tends to $0$ uniformly as $|x|\to 1$. Then the function $$g(t) = t \exp\left[\int_0^{\mu(t)}\gamma \Upsilon(x) dx\right],$$ is decreasing on the interval $(0, t_\circ)$, where $\gamma=\alpha (n-1)^2$, $\Upsilon$ is defined in \eqref{newper}  and $t_\circ = \max_{x\in\mathbb{B}} u(x)$.
\end{theorem}

If  $f(x) \equiv 1$, the function $g$ turns out to be constant and this is an important property of $g$.

The proof of this theorem is mostly based on the methods developed in \cite{3}, translated from the Euclidean to the hyperbolic setting, and in \cite{kulik} translated from the planar case to the higher-dimensional case. 

\begin{proof}[Proof of Theorem \ref{monotone}]
We start with the coarea formula
$$\mu(t)=\tau(A_t)=\int_{A_t}\frac{2^n}{(1-|x|^2)^n} dx=\int_t^{\max u}\int_{|u(x)|=\kappa}\frac{2^n|\nabla u|^{-1}}{(1-|x|^2)^n} d\mathcal{H}^{n-1}(x) d\kappa.$$ Then we get
\begin{equation}\label{derlevel}
-\mu'(t) = \int_{u = t}|\nabla u|^{-1}\frac{2^nd\mathcal{H}^{n-1}(x) }{(1-|x|^2)^n}
\end{equation}
along with the claim that $\{x:u(x) = t\} = \partial A_t$ and that this set is a smooth hypersurface for almost all $t\in (0, t_\circ)$. Here $dS=d\mathcal{H}^{n-1}$ is $n-1$ dimensional Hausdorff measure. Observe that a similar formula has been proved in \cite{kulik}.
These assertions  follow the proof of Lemma $3.2$ from \cite{3}. We point out that, since $u$ is real analytic and non-constant, then it is a well-known fact from measure theory that the level set $\{x:u(x) = t\}$ has a zero measure (see \cite{zero}), and this is equivalent to the fact that the $\mu$ is continuous.

Following the approach from \cite{3, kulik}, our next step is to apply the Cauchy--Schwarz inequality to the hyperbolic area of $\partial A_t$:
\begin{equation}\label{CS}\begin{split}|\partial A_t|_h^2 &= \left(\int_{\partial A_t} \frac{2^{n-1}dS}{(1-|x|^2)^{n-1}}\right)^{2}\\&\le \int_{\partial A_t} |\nabla u|^{-1} \frac{2^{n}dS}{ (1-|x|^2)^n}\int_{\partial A_t} \frac{|\nabla u| 2^{n-2}dS}{(1-|x|^2)^{n-2}}.\end{split}
\end{equation}

Let $\nu=\nu(x)$ be the outward unit normal to $\partial A_t$ at a point $x$. Note that, $\nabla u$ is parallel to $\nu$, but directed in the opposite direction. Thus we have $|\nabla u| = -\left<\nabla u,\nu\right>$. Also, we note that since for $x\in \partial A_t$ we have $u(x) = t$, we obtain for $x\in \partial A_t$ that
$$\frac{|\nabla u(x)|}{t} = \frac{|\nabla u(x)|}{u} =  \left<\nabla  \log u(x), \nu\right>.$$

Now the second integral on the right-hand side of \eqref{CS} can be evaluated by the Gauss's divergence theorem:
\[\begin{split}\int_{\partial A_t} \frac{|\nabla u| |dS|}{(1-|x|^2)^{n-2}}&=-t  \int_{A_t}\mathrm{div}\left(\frac{\nabla\log u(x)}{(1-|x|^2)^{n-2}}\right) dx
\\&=-t  \int_{A_t}\frac{1}{(1-|x|^2)^n}\Delta_h {\log u(x)} dx.\end{split}\]
Now we plug
$u= v(x)^a\Phi^\alpha_n(x)$, where $v(x)=|f(x)|$, and calculate $$-t \Delta_h \log(v^a \Phi^\alpha_n)=-(a t\Delta_h \log v+t\alpha \Delta_h \log \Phi_n)\le 0+4t \gamma, $$ where
$\gamma=\alpha (n-1)^2.$ By using \eqref{derlevel} and \eqref{CS} we obtain
\[\begin{split}|\partial A_t|_h^2 & \le  (-\mu'(t))\int_{\partial A_t} \frac{|\nabla u| 2^{n-2}dS}{(1-|x|^2)^{n-2}}.
\\&\le -2^{n-2}\cdot 4 t \gamma\mu'(t) \mu(t)/2^n\\&=- t \gamma \mu'(t)\mu(t). \end{split}\]
So by \eqref{newper1}, we have
\begin{equation}\label{ypsibaraz} t \gamma \mu'(t)\mu(t)+\frac{\mu(t)}{\Upsilon(\mu(t))}\le 0\end{equation}
with  equality in \eqref{ypsibaraz}  if $v$ is a constant because in that case $A_t$ is a ball centered at the origin.

Thus
$$G(t)=-\int_t^{t_\circ}  \gamma \mu'(t)\Upsilon(\mu(t))dt -\int_t^{t_\circ}\frac{1}{t}dt = \int_0^{\mu(t)} \gamma \Upsilon(x)dx -\log \left(\frac{t_\circ}{t}\right) $$ is non-increasing. 
Therefore $$g(t):=\exp(G(t)) =t\exp\left[\int_{0}^{\mu(t)} \gamma \Upsilon(x)dx \right]$$ is non-increasing.

\end{proof}
\begin{remark}
{Note that for the function $f(x) \equiv 1$ everywhere in the proof above we have  equalities for all values of $a$ and $\alpha$.}
\end{remark}
\section{Weak-type estimate for the $\mathcal{M}-$Hardy class $\mathbf{h}^p$ and the proof of Theorem \ref{Hardy}}\label{sec3}
In this section, we are going to prove the following bound for the measure of the so-called superlevel sets of  functions from the Hardy spaces. Theorem \ref{Hardy} is then an easy matter. In what follows we keep the same notation as in the previous section.

\begin{theorem}\label{kalajda}
Let $f\in \mathbf{h}^p$ for $p>0$ with $\|f\|_{h^p}=1$ and put $u(x) = |f(x)|^p\Phi_n(x)$. Then for all $t\in (0, \infty)$ we have
\begin{equation}\label{hardy bound}
\mu(t) \le \mu_1(t),
\end{equation}
where $$\mu(t)= |\{x: u(x)\ge t\}|_h\ \ \text{   and   } \mu_1(t) = |\{x: \Phi_n(x)\ge t\}|_h.$$
\end{theorem}

Note that this theorem extends the corresponding result in \cite{kulik}, where \cite[Conjecture~$2$]{7} is verified. Indeed it is easy to check that
$$ |\{x: \Phi_n(x)\ge t\}|_h=4\pi\max\{1/t-1,0\},$$ for $n=2$ which coincides with the corresponding result of Kulikov in \cite{kulik} after normalization.
First, we prove the following lemma.
\begin{lemma}\label{hardy2023}
Theorem~\ref{kalajda} and Theorem~\ref{Hardy} are true provided that $f$ is bounded.
\end{lemma}
\subsection{Proof of Theorem~\ref{kalajda} for bounded functions}\label{subkala}
Put $t_\circ = \max_{x\in\mathbb{B}} u(x)$. This number is well-defined since $f$ is bounded. Assume without loss of generality that $p>1$. Indeed, if $p>0$ and if $f$ is a positive log- $\mathcal{M}-$subharmonic function such that $f\in \mathbf{h}^p$, then $g=f^{p/2}$ is positive log- $\mathcal{M}-$subharmonic function such that $g\in \mathbf{h}^2$.
In particular, for $t\ge t_\circ$ the bound \eqref{hardy bound} holds trivially.

Assume that there exists some $0 < t_1 < t_\circ$ such that $\mu(t_1) > \mu_1(t_1)$. Then $\mu(t_1) = \mu_1(t_1/c) $ for some $c > 1$, because $\lim_{s\to 0}\mu_1(s)=+\infty$. We claim that in that case for all $0 < t < t_1$ we have $\mu(t) \ge  \mu_1(t/c)$.

Indeed, by applying the pointwise bound together with $u(x)\to 0$ as $|x|\to 1$, we see that Theorem \ref{monotone}
can be applied to $f$ with $a = p, \alpha = 1$, and we get that  $$g(t) = t \exp\left[\int_0^{\mu(t)}\gamma \Upsilon(x) dx\right],$$ is decreasing. Since

\[\begin{split}g(t_1) &=t_1 \exp\left[\int_0^{\mu(t_1)}\gamma \Upsilon(x) dx\right]
\\&=t_1 \exp\left[\int_0^{\mu_1(t_1/c)}\gamma \Upsilon(x) dx\right]
\\&=c\cdot t/c \exp\left[\int_0^{\mu_1(t/c)}\gamma \Upsilon(x) dx\right]
\\&<g(t)=t\exp\left[\int_0^{\mu(t)}\gamma \Upsilon(x) dx\right].\end{split}\]
From this we get  $\mu(t) \ge  \mu_1(t/c)$.

 Now we are going to use Proposition~\ref{propozicija}, which says that  $ \|f \|_{r, pr}^{pr} \to  \|f \|_{p}^p = 1$ as $r\to 1^+$. Note that we can express the $\mathbf{B}^{pr}_r$ norms through the function  $\mu(t)$:
\begin{equation}\label{form} \|f \|_{r,pr}^{pr} = c_r \int_0^{t_\circ}\mu(t)t^{r-1}dt,\end{equation}
where $c_r =c(r)$ is defined in \eqref{calpha}, and it satisfies the relation $c_r\int_0^1 \mu_1(r)t^{r-1}dt = 1$. We now use the fact that $c_r\to 0$ as $r\to 1$ (see \eqref{caz}). By the formula \eqref{form} and above bound we have
\begin{equation}\label{temp3}
 \|f \|_{r, pr}^{pr} \ge c_r \int_0^{t_1} (\mu_1(r/c))t^{r-1}dt=c_r c \int_0^{t_1/c} (\mu_1(s))s^{r-1}ds.
\end{equation}
On the other hand
\[\begin{split}1 &= c_r\int_0^1 \mu_1(t)t^{r-1}dt \\&= c_r\int_0^{t_1/c} \mu_1(t)t^{r-1}dt + c_r\int_{t_1/c}^1 \mu_1(t)t^{r-1}dt = P(r)+Q(r).\end{split}\]

Since $c_r\to 0$ as $r\to 1$, we have that $Q(r)\to 0$ as $r\to 1$ because the function we are integrating is bounded. Therefore, $P(r)\to 1$ as $r\to 1$. On the other hand, the right-hand side of \eqref{temp3} is at least $cP(r)$. Therefore $1 =\lim_{r\to 1} \|f \|_{r, pr}^{pr} \ge c\lim_{r\to 1}P(r)=c$ which is a contradiction. Recall that $c>1$.

\subsection{Proof of Theorem~\ref{Hardy} for bounded functions.}\label{subha} \ref{subha}
As in \cite{kulik},  we can assume that $\lim_{t\to 0^{+}}G(t)  = 0$. Then this integral can be expressed through the function  $\mu(t)$ as
\begin{equation}\label{dG}
\int_0^\infty \mu(t)dG(t).
\end{equation}
Note that here we used that the function $\mu(t)$ is continuous, that is the sets $\{x\in\B :u(x) = t\}$, $t>0$,  have zero measure.

Since $G$ is increasing, measure $dG(t)$ is positive. Thus, by Lemma~\ref{hardy2023}, \eqref{hardy bound} this integral is at most
$$\int_0^\infty \mu_1(t)dG(t),$$
which is the value of \eqref{hardyineq} for $f(x)\equiv 1$.

\begin{proof}[Proof of Theorem~\ref{kalajda} and Theorem~\ref{Hardy}]
We repeat the proof of Lemma~\ref{hardy2023}. We also have $1 =\lim_{r\to 1} \|f \|_{r, pr}^{pr}$, because of Proposition~\ref{propozicija1} below. This yields Theorem~\ref{kalajda}.
Now we repeat the proof of Theorem~\ref{Hardy} for bounded functions, by using Theorem~\ref{kalajda}, instead of  Lemma~\ref{hardy2023},  and get Theorem~\ref{Hardy} for unbounded function $f$.
\end{proof}
\section{Proof of Theorem \ref{Bergman}}\label{sec4}
As in \cite{kulik}, we restrict ourselves to the case $\lim_{t\to 0^{+}}G(t)  = 0$.
Let $\mu(t) = \tau(\{ x: u(x) > t\})$ where $u(x) = |f(x)|^p(\Phi_n(x))^\alpha$.
Applying Theorem \ref{monotone} to $f$ with $a = p$, we get that the function
$$g(t) = t \exp\left[\int_0^{\mu(t)}\gamma \Upsilon(x) dx\right],$$
is decreasing on $(0, t_\circ)$ with $t_\circ = \max_{x\in \mathbb{D}}u(x)$. Proposition~\ref{may2023} ensures  the existence of $t_\circ$.

For $f\equiv 1$, $g$ is a constant function equal to $1$.

Let $$\Lambda(u) = \int_0^u \gamma \Upsilon(x) dx$$ and $\Theta=\Lambda^{-1}$. Note that $\Theta$ is increasing. Then $$\mu(t) = \Theta\left(\log\frac{g(t)}{t}\right).$$
We assume that $ \|f \|_{\alpha,p} = 1$, that is
$$I_1=\int_0^{t_\circ} \mu(t)dt =\int_0^{t_\circ} \Theta\left(\log\frac{g(t)}{t}\right)dt= \frac{1}{c(\alpha)}.$$
Now the integral in \eqref{bergmanineq} can be rewritten as
$$I_2=\int_0^{t_\circ} \Theta\left(\log\frac{g(t)}{t}\right)G'(t)dt.$$
Then by  Lemma~\ref{lemm} below, by taking $\Phi(s) = \Theta(\log(s))$ and $\Psi(t)=G'(t)$,  the maximum of $I_2$ under $I_1=\frac{1}{c(\alpha)}$ is attained for $g\equiv 1$.
\begin{lemma}\label{lemm}
Assume that $\Phi, \Psi$ are positive increasing functions and $g$ is a positive non-increasing function such that
$$\int_0^{t_\circ} \Phi\left({g(t)}/{t}\right)dt = \int_0^{t_\circ} \Phi\left({1}/{t}\right)dt=c.$$

Then $$\int_0^{t_\circ} \Phi\left({g(t)}/{t}\right)\Psi(t)dt\le \int_0^{t_\circ} \Phi\left({1}/{t}\right)\Psi(t)dt.$$
\end{lemma}
\begin{proof}

Choose $a\in [0,t_\circ]$ such that $g(t)\geq 1$ for $t\leq a$ and $g(t)\leq 1$ for $t\geq a$. Then $$\chi(t):=(\Phi(g(t)/t) -\Phi(1/t))(\Psi(t)-\Psi(a))\leq 0$$ for all $t\in [0,t_\circ] $. By integrating $\chi(t)$ for $t\in(0,t_\circ)$ we get
\[\begin{split}
\int_0^{t_\circ}\Phi\left({1}/{t}\right) \Psi(a) &-\Phi\left({g(t)}/{t}\right) \Psi(a)-\Phi\left({1}/{t}\right) \Psi(t)+\Phi\left({g(t)}/{t}\right) \Psi(t)dt
\\&= \Psi(a)\int_0^{t_\circ} \left(\Phi\left({g(t)}/{t}\right)-\Phi\left({1}/{t}\right)\right)dt
\\&+\int_0^{t_\circ} \Psi(t)\left(\Phi\left({g(t)}/{t}\right)-\Phi\left({1}/{t}\right) \right)dt\le 0.
\end{split}\]
Since $$\int_0^{t_\circ} \left(\Phi\left({g(t)}/{t}\right)-\Phi\left({1}/{t}\right)\right)dt=0,$$ the result follows.
\end{proof}

\section{Some applications for the case $n=2$}\label{sec5}

Let ${B}^2_{2/p}$, $1<p<2$  be the space of harmonic functions in the unit disk $\mathbb{D}$  such that  $$\|f\|^2_{2/p,2}:=\int_{\mathbb{D}}|f(z)|^2 (1-|z|^2)^{2/p-1}\frac{dxdy}{\pi}<\infty$$ and let $h^p$ be the standard Hardy space with the norm  $$\vertiii{ f }_{{h}^p}^2:=\|\sqrt{|a|^2+|b|^2}\|^2_{p}.$$
Assume that $f=a+\bar b$ is a harmonic function defined in the unit disk, where $a$ and $b$ are holomorphic functions. Then $\log (|a|^2+|b|^2)$ is subharmonic (see e.g. \cite{kalaj}).

Then we obtain a special case of contraction from the Hardy space to a Bergman space is $({h}^p, \vertiii{ \cdot }) \subset {B}^2_{2/p}$ for $1 < p < 2$, which extends a corresponding result in \cite{kulik} and a classical result \cite{5} and also the recent results \cite{11}. Namely for $$f(z) = a+\bar b=\sum_{n = 0}^\infty a_nz^n+\sum_{n = 0}^\infty \overline{b_n}\bar z^n\in {B}^2_{2/p},$$ where $b_0=0$,
 we can express its norm as follows
$$ \|f \|_{2/p,2}^2 = \sum_{n = 0}^\infty \frac{|a_n|^2+|b_n|^2 }{c_{2/p}(n)},\ \  c_{2/p}(n) = \binom{n + 2/p -1}{n}.$$
Thus, for a function $f\in {h}^p$, from Corollary~\ref{shtune}, we have
\begin{equation}\label{coeff}
\sum_{n = 0}^\infty \frac{|a_n|^2+|b_n|^2}{c_{2/p}(n)} \le  \vertiii{ f }_{{h}^p}^2.
\end{equation}

Now by \cite[Theorem~2.1]{kalaj}, for $f\in {h}^p$, $b_0= 0$,  for $1<p\le 2$, we obtain the following inequality

\begin{equation}\label{coeff1}
\sum_{n = 0}^\infty \frac{|a_n|^2+|b_n|^2}{c_{2/p}(n)} \le \frac{1}{1-|\cos{\pi}/{p}|}\|f\|^2_{p},
\end{equation}
where $$\|f\|^p_{p}=\int_{\partial\mathbb{D}}|f(\zeta)|^p\frac{|d\zeta|}{2\pi}.$$

By using the result from \cite{vesna}, that the Jacobian of a harmonic diffeomorphism is log-superharmonic, and the fact that the Jacobian $J(f,z)=|a'(z)|^2-|b'(z)|^2$ is real analytic,  Corollary~\ref{shtune}  implies
\begin{corollary}
Let $1<p<2$ and $\alpha>1$. Assume that $f$ is a harmonic diffeomorphism of the unit disk $\mathbb{D}$ onto a two-dimensional domain $\Omega$. Assume further that $1/J_f(z) \in h_{p/\alpha}$. Then we have
$$\left((\alpha-1)\int_{\mathbb{D}}\frac{1}{J^p_f(z)}(1-|z|^2)^{\alpha-2}\frac{dx dy}{\pi}\right)^{1/p}\le \|1/J_f(z) \|_{{p/\alpha}}.$$
\end{corollary}
Since $\log (|a|^2+|b|^2)$ is subharmonic, if $a$ and $b$ are analytic functions (see property 3 of the admissible monoid), because of Corollary~\ref{shtune},  we obtain the following Corollaries
\begin{corollary}\label{co32}
Assume that $f=a+\bar b: \mathbb{D}\to \mathbb{C} $ is a harmonic mapping and that $\alpha>1, p>1$. Then $$\left((\alpha-1)\int_{{\mathbb{D}} } (|a|^2+|b|^2)^{p/2}(1-|z|^2)^{\alpha-2}\frac{dx dy}{\pi}\right)^{1/p}\le \|(|a|^2+|b|^2)^{1/2} \|_{{p/\alpha}}.$$
\end{corollary}

\begin{corollary}\label{mundial} For $p>1$ and a harmonic function $f\in {h}^p$ we have the following isoperimetric type  inequality
\begin{equation}\label{isoi}\|f\|_{B^{2p}}\le C_p\|f\|_{p},\end{equation} where $$C_p= \frac{\sqrt{2}\cos\frac{\pi}{4p}}{\left(1-\abs{\cos\frac{\pi}{p}}\right)^{1/2}}.$$Here $B^{2p}$ is the Bergman space of harmonic mappings belonging to the Lebesgue space $\mathcal{L}^{2p}(\mathbb{D})$.
\end{corollary}
Observe that for $p\ge 2$, $$C_p=\frac{1}{2}\csc\frac{\pi}{4p}.$$ The inequality \eqref{isoi} for $p>1$ being an integer has been proved by the author in  \cite[Theorem~2.11]{kalaj}. We also refer to a recent improvement of \eqref{isoi} in \cite{petar} for $p>2$, which is based on \eqref{isoi} for $p=2$  and Jensen's inequality.  For $p\in(1,2]$ we have, $$C_p=\cos\frac{\pi }{4 p} \sec\frac{\pi }{2 p}.$$ The inequality \eqref{isoi} for such constant $C_p$ has been  proved for real-valued harmonic functions in \cite{elka}. In this case we do not have such a restriction.
\begin{proof}[Proof of Corollary~\ref{mundial}] Let $f(z)=a(z)+\overline{b(z)}$ and assume without loss of generality that $a(0)=0$.
By integrating \cite[eq.~2.3]{kalaj} in the interval $[0,1]$ we get $$\left(\int_{\mathbb{D}}|a+\bar b|^{2p}\frac{dxdy}{\pi}\right)^{\frac{1}{2p}}\le I:=\sqrt{2}\cos\frac{\pi}{4p}\left(\int_{\mathbb{D}}(|a|^2+|b|^2)^p\frac{dxdy}{\pi}\right)^{\frac{1}{2p}}.$$  Then by Corollary~\ref{co32}, by choosing $\alpha=2$, we get $$I\le J:=\sqrt{2}\cos\frac{\pi}{4p}\||a|^2+|b|^2\|_{h_p}.$$ Now  \cite[Theorem~2.1]{kalaj} implies $$J\le \frac{\sqrt{2}\cos\frac{\pi}{4p}}{\left(1-\abs{\cos\frac{\pi}{p}}\right)^{1/2}}\|f\|_{p}.$$
The result follows.
\end{proof}

Assume now that $f:\mathbb{D}\to\Sigma \subset\mathbb{R}^n$ is a conformal parameterization of the minimal surface $\Sigma$.
Since $\log |f_x(z)|^2=\log \left(|\mathbf{p}(z)|(1+|\mathbf{q}(z)|^2)\right)$, where $\mathbf{p}$ and $\mathbf{q}$ are holomorphic functions in $z=x+i y$ (the so-called Enneper -Weierstrass parameters), we have the following corollary.
\begin{corollary}
Assume that $\alpha>1, p>1$ and  $f: \mathbb{D}\to \mathbb{R}^n$ is the Enneper -Weierstrass parameterisation of a minimal surface in $\mathbb{R}^n$. Then
\begin{equation}\label{onedimi}\left((\alpha-1)\int_{\mathbb{D}} |f_x(z)|^{p}(1-|z|^2)^{\alpha-2}\frac{dx dy}{\pi}\right)^{1/p}\le \ \|f_x \|_{{p/\alpha}}.\end{equation}
 For $p=\alpha=2$, the above formula is simply the isoperimetric inequality for minimal surfaces.
\end{corollary}

\section{Appendix}
\begin{proposition}\label{pbelo} Let $u(x)=g(|x|)$. Then a solution of $\Delta_h \log u=-4b$ which is differentiable at $x=0$ is given by the formula \begin{equation}\label{ub}u_b(x)=\exp\left\{\frac{b (2-n) r^2}{(n-1) n} F\left[\begin{array}{c}
                                           1,1,2-\frac{n}{2} \\
                                           2,1+\frac{n}{2}
                                         \end{array}; r^2\right]\right\} \left(1-r^2\right)^{\frac{b}{ (n-1)}}.\end{equation}  Then for $b=(n-1)^2$ we define $\Phi_n(r)=u_b(x)$ and have \begin{equation}\label{esfi}E_n\left(1-r^2\right)^{n-1}\le \Phi_n(r)\le \left(1-r^2\right)^{n-1}\end{equation} and inequality is strict for $n>2$ and $r>0$. Here $$E_n=\exp\left\{\frac{(n-1) (2-n) }{ n} F\left[\begin{array}{c}
                                           1,1,2-\frac{n}{2} \\
                                           2,1+\frac{n}{2}
                                         \end{array}; 1\right]\right\} .$$
                                         \end{proposition}

\begin{proof} The differential equation $\Delta_h \log u=-4b,$ reduces to
$$\left(1-r^2\right) \left(\frac{\left(2 (-2+n) r^2+(-1+n) \left(1-r^2\right)\right) h(r)}{r}+\left(1-r^2\right) h'(r)\right)=-4b,$$ where  $h(r)=g'(r)$,  $g(r) = \log u(|x|)$ and $r=|x|$. Then the general solution of the last equation  is given by $$h(r)=\frac{r^{1-n} \left(1-r^2\right)^{-2+n} \left(n c-4b r^n F_0\left[\frac{n}{2},n,\frac{2+n}{2},r^2\right]\right)}{n}.$$ Here $F_0$ is the Gaussian hypergeometric function. The solution which  is defined in $x=0$ is given by $$h(r)=\frac{-4b \left(1-r^2\right)^{-2+n} r F_0\left[\frac{n}{2},n,\frac{2+n}{2},r^2\right]}{n}=\frac{-4br F_0\left[1,1-\frac{n}{2},1+\frac{n}{2},r^2\right]}{n (1 - r^2)}.$$
 The last equality follows from Euler transformations.

Then after straightforward computations we obtain  $$g(r) = \int_0^r h(t) dt=-b\frac{(-2+n) r^2 F\left[\begin{array}{c}
                                           1,1,2-\frac{n}{2} \\
                                           2,1+\frac{n}{2}
                                         \end{array}; r^2\right]-n \log\left[1-r^2\right]}{ (-1+n) n}.$$ and $u(x) = \exp(g(r))$ is given by \eqref{ub}.

                                 To prove the \eqref{esfi} we need to prove the monotonicity  of
                                         $$\phi(r):=r^2\frac{(n-1) (2-n) }{ n}  F\left[\begin{array}{c}
                                           1,1,2-\frac{n}{2} \\
                                           2,1+\frac{n}{2}
                                         \end{array}; r^2\right].$$ First we have $$\phi'(r)=2 \frac{(n-1) (2-n) }{ n} r F_0\left[1,2-\frac{n}{2},\frac{2+n}{2},r^2\right],$$ then observe that

$$F_0[a,b,c,z]=\frac{1}{\mathrm{B}[b,c-b]}\int_0^1 x^{b-1}(1-x)^{c-b-1}(1- zx)^{-a}dx,$$ where $\mathrm{B}$ is the beta function. This functions is clearly positive for $c=\frac{2+n}{2}$, $b=2-\frac{n}{2}$, $a=1$ and $z=r^2<1$, and thus $\phi$ is decreasing. This implies \eqref{esfi}.

\end{proof}

The point evaluations are continuous functionals in ${h}^p$ and ${B}_\alpha ^p$. This is the content of the following proposition.
\begin{proposition}\label{propo1}
Let $p>1$ and $\alpha>1$. There are two constants $C_1=C_1(n)$ and $C_2=C_2(n,\alpha)$ such that, if $f$ is $\mathcal{M}-$subharmonic and belongs to the Hardy space $\mathcal{S}^p$, then
\begin{equation}\label{prima}|f(x)|^p\Phi_n(x)\le C_1\|f\|_p^p\end{equation} and if $f\in {B}_\alpha^p$ then

\begin{equation}\label{seconda}|f(x)|^{p}\Phi_n^\alpha(x)\le C_2\|f\|^p_{\alpha, p}.\end{equation}
Moreover, if  $f\in S^p$ then \begin{equation}\label{vanishH}\lim_{|x|\to 1}|f(x)|^p\Phi_n(x)=0.\end{equation}
\end{proposition}
\begin{proof}
First of all relation \eqref{prima} follows from \eqref{esfi} and \cite[Lemma~7.2.1]{stoll}.

To prove \eqref{seconda} we proceed similarly. This time we make use of \eqref{esfi} and \cite[Eq.~10.1.5]{stoll}, which is formulated for $\mathcal{M}-$harmonic functions, but the same proof can be applied for the $\mathcal{M}-$subharmonic functions.

Prove now \eqref{vanishH}. Let $f$ be a positive $-\mathcal{M}-$subharmonic function $f$ on the unit ball belonging to the Hardy space $\mathcal{S}^p$.

Let  $F=P_h[\hat{f}](x)$ be the least harmonic majorant of $f$ (\cite[Theorem~7.1.1]{stoll}), which is in $\mathcal{S}^p$. Now  we define  the mapping $F_r(x)=P_h[f_r](x)$, where $f_r(x)=F(rx)$, $x\in \mathbb{S}$ and $0<r<1$. We claim that for all $\epsilon> 0$ there exists $r < 1$ such that such that \begin{equation}\label{need}\|F_r-F\|_p\le \epsilon.\end{equation}
Since $\hat{f}\in L^p(\mathbb{S})$, there exists a continuous function $g$ in $\mathbb{S}$, such that $\|g-\hat{f}\|_p<\epsilon/3$. Let $G(x)=P_h[g](x)$ and $G_r(x)=P_h[g_r]$, where $g_r(x)=G(rx)$, $x\in \mathbb{S}$, $0<r<1$. Then $G_r-F_r$ is $\mathcal{M}-$harmonic and so $|G_r-F_r|^p$ is $\mathcal{M}-$subharmonic for $p\ge 1$. The same hold for $G-H$ and $|G-H|^p$. Then by \cite[Theorem~5.4.2]{stoll} we have
$$\|G-F\|_p=\|P_h[\hat f-g]\|_p\le  \|\hat f-g\|_p\le \epsilon/3$$ and

\[\begin{split}\|G_r-F_r\|_p&=\|P_h[f_r-g_r]\|_p\\&\le  \|f_r-g_r\|_p\\&=\left(\int_{\mathbb{S}}|F(r\eta)-G(r\eta)|^pd\sigma(\eta)\right)^{1/p}\\&\le \left(\int_{\mathbb{S}}|\hat f(\eta)-g(\eta)|^pd\sigma(\eta)\right)^{1/p} = \|\hat f-g\|\le \epsilon/3.\end{split}\]
In the last inequality we used \cite[Theorem~5.4.2]{stoll}, which states that $\int_{\mathbb{S}}k(r\eta)d\sigma(\eta)$ is an increasing function for $r\in[0,1)$, provided that $k$ is $\mathcal{M}-$subharmonic.

Thus
\[\begin{split}\|F-F_r\|_p&\le \|F_r-G_r\|_p+\|G_r-G\|_p+\|G-F\|_p\\&\le \|g-\hat{f}\|_p+\|g_r-g\|_p+\|g-\hat{f}\|_p\\&\le \frac{2\epsilon}{3}+\|g_r-g\|_p.\end{split}\]

 Now by \cite[Corollary~5.3.4]{stoll}, because $g$ is continuous on $\mathbb{S}$, there is $r<1$ such that $\|g_r-g\|_p<\epsilon/3$. This implies \eqref{need}.

 Then we have
$$|F(x)|\le |F(x)-F_r(x)|+|F_r(x)|.$$
From \eqref{prima} and \eqref{need}, we get $$|F(x)-F_r(x)|^p \Phi_n(x)\le C\epsilon^p.$$
On the other hand, because $F_r$ is smooth up to the boundary we have $$\lim_{|x|\to 1}|F_r(x)|^p  \Phi_n(x)=0.$$ So $$\limsup_{|x|\to 1}|F(x)|^p  \Phi_n(x)\le  C\epsilon^p.$$
Because $F$ is the harmonic majorant of $f$, by the previous relation we get
\[\begin{split}
\limsup_{|x|\to 1}|f(x)|^p \Phi_n(x)&\le \limsup_{|x|\to 1}|F(x)| \Phi_n(x)\le C\epsilon^p.\end{split}\]
Thus $$\limsup_{|x|\to 1}|f(x)|^p \Phi_n(x)=0$$ as was stated.
\end{proof}
Now we prove also an analog to the vanishing condition \eqref{vanishH} for the Bergman spaces.
\begin{proposition}\label{may2023} If $f$ is $\mathcal{M}$-subharmonic and
$$\int_{\mathbb{B}}|f(x)|^{p} \Phi_n^\alpha(x)d\tau(x) < \infty$$
then $$|f(x)|^p \Phi_n^\alpha(x)$$  is bounded and tends to $0$ uniformly as $|x|\to 1$.
\end{proposition}
We start with a lemma
\begin{lemma} If f is $\mathcal{M}$-subharmonic in $\mathbb{B}$ then from \cite[Theorem 4.3.5]{stoll}:
\begin{equation}\label{eqmaj1}|f(0)|\le   C(n)\int_{|x|<1/2} |f(x)| dx.
\end{equation}
\end{lemma}
By using the previous lemma and H\"older's inequality we have
\begin{lemma} If $f$ is $\mathcal{M}$-subharmonic in $\mathbb{B}$  and $1 < p < \infty$ then
\begin{equation}\label{eqmaj}|f(0)|^p\le   C(n,p)\int_{|x|<1/2} |f(x)|^p dx.
\end{equation}
\end{lemma}

\begin{proof}[Proof of Proposition~\ref{may2023}] We already know that $|f(0)|^p$ is bounded. Now, we will use the $\mathcal{M}-$invariance of
our integral to translate this bound from $0$ to any other point. To prove that it tends to
$0$, observe that by the standard measure theory, we have
$$\lim_{r\to 1}\int_{r<|x|<1}|f(x)|^{p} \Phi_n^\alpha(x)d\tau(x)=0.$$
Now, pick $s > r$ so that for all $x_0$ with $|x_0| > s$ when we consider the M\"obius transformation
$\varphi=\varphi_{x_0}$ ($\varphi^{-1}=\varphi$) which  sends $0$ to $x_0$ then the ball of radius $1/2$ around $0$ is getting mapped into the
subset of a set $r < |y < 1$. We can take for example $s=s(r)=\frac{r+1/2}{1+r/2}$, because $$|y|=|\varphi(x)|=\frac{|x-x_0|}{[x,x_0]}\ge \frac{|x_0|-|x|}{1-|x_0||x|}>r$$ if $|x_0|>s$ and $|x|<1/2$.  Here $$[x,x_0]^2:=1+|x|^2|x_0|^2-2\left<x,x_0\right>\ge (1-|x||x_0|)^2.$$ Then by the translated pointwise bound we
get
\[\begin{split}|f(x_0)|^p&=|f(\varphi(0))|^p\\&\le C(n,p)\int_{|x|<1/2}|f(\varphi(x))|^pdx\\&=\int_{\varphi(|x|<1/2)}|f(y)|^p J_\varphi(y)dy,\end{split}\] where $J_\varphi(y)$ is the Jacobian of $\varphi$ at $y$. Now we use the fact that $\varphi(|x|<1/2)\subset \{|y|: r<|y|<1\}$ and $$J_\varphi(y)=\frac{(1-|x_0|^2)^n}{[y,x_0]^{2n}}.$$  Now because  ${(1-|x_0|)}\le {[y,x_0]},$ we have

\[\begin{split}(1-|x_0|^2)^{(n-1)\alpha} J_\varphi(y)&=\frac{(1-|x_0|^2)^{{(n-1)\alpha} +n}}{[y,x_0]^{2n}}\\&=\frac{(1-|x_0|^2)^{{(n-1)\alpha} +n}}{[y,x_0]^{{(n-1)\alpha} +n-(n-1)\alpha +n}}\\& \le C_n(1-|y|^2)^{(n-1)\alpha-n}. \end{split}\]
Thus
\[\begin{split}|f(x_0)|^p \Phi_n^\alpha(x_0)&\le C(p,n)\int_{\varphi(|x|<1/2)}(1-|x_0|^2)^{(n-1)\alpha}|f(y)|^pJ(y,\varphi)dy\\&
\le C_1(p,n)\int_{r<|y|<1}\Phi_n^{\alpha}(y) |f(y)|^p d\tau(y).
\end{split}\]
The last quantity tends to zero as $r\to 1$ and this implies the claim.
\end{proof}

\begin{proposition}\label{propozicija} Let $p>0$ and $\alpha>1$. Then \begin{equation}\label{inc} \mathbf{h}^p\subset \mathbf{B}_\alpha^{\alpha p}.\end{equation} Furthermore assume that $f\in\mathbf{h}^p$. Then
\begin{equation}\label{betap}\lim_{\alpha\to 1+}\|f\|_{\alpha, p}=\|f\|_{p}.\end{equation}
Moreover, if $f$ is bounded, then
\begin{equation}\label{alphap}\lim_{\alpha\to 1+}\|f\|_{\alpha, \alpha p}=\|f\|_{p}.\end{equation}
\end{proposition}

\begin{proof} We can assume that $p>1$, otherwise we consider the function $g=|f|^{p/2}$ and apply the following proof. To prove \eqref{inc} assume that $f\in \mathbf{h}^p$ and let $g$ be the last $\mathcal{M}-$harmonic majorant of $f$.
We use  the Hardy-Littlewood type inequality for  $\mathcal{M}-$harmonic functions in the unit ball by Stoll: \cite[Theorem~10.6.3]{stoll} (see also \cite{flett}) \begin{equation}\label{eqpo1}\|g\|_{\alpha, \beta p}
\le C(p,\alpha,n)\|g\|_{p}.\end{equation} Since $\|g\|_{p}=\|f\|_{p}$ and $\|f\|_{\alpha, \alpha p}\le \|g\|_{\alpha,\alpha p}$, it follows that $f\in \mathbf{B}_\alpha^{\alpha p}$ and this concludes \eqref{inc}.

From  \eqref{posexp} we know that $\Delta_h f\ge 0$, and by corresponding theorem \cite[Theorem~5.4.2]{stoll}, we get \begin{equation}\label{ec1}\|f\|^p_{p}=\lim_{r\to 1} \int_{\mathbb{S}}|f(r\zeta)|^pd\sigma(\zeta).\end{equation}
Then from \eqref{meanconv} and \eqref{ec1} we have

\begin{equation}\label{david}\begin{split}\|f\|_{\alpha, p}^p&=c(\alpha)\int_{\mathbb{B}}|f(x)|^p\Phi^\alpha_n(|x|)  (1-|x|^2)^{-n}\frac{dV(x)}{{\omega_n}}\\
&={c(\alpha)}\int_0^1 \Psi(r) dr\int_{\mathbb{S}}|f(r\zeta)|^p d\sigma(\zeta)
\\&\le \|f\|^p_{p},\end{split}\end{equation} here $\alpha>1$, $\Psi(r)=r^{n-1}\Phi_n^\alpha(r)(1-r^2)^{-n}$ and $$c(\alpha) \int_{0}^1 \Psi(r) dr=1.$$
On the other hand \eqref{david} implies  $$\limsup_{\alpha\to 1^+}\|f\|_{\alpha, p}\le \|f\|_{p}.$$
Furthermore,  because $$\|f\|^p_{p}=\lim_{r\to 1} \int_{\mathbb{S}}|f(r\zeta )|^p{d\sigma(\zeta)},$$ for every $\epsilon>0$ there exists $r_1<1$ such that
for $r\in(r_1, 1)$ we have
\begin{equation}\label{hardep}\int_\mathbb{S}|f(r\zeta )|^p{d\sigma(\zeta)}\ge \|f\|^p_{p}-\epsilon.\end{equation}
Then  \[\begin{split}\|f\|^p_{\alpha, p}&=c(\alpha)  \int_0^1 r^{n-1} \Phi_n^\alpha(r)(1-r^2)^{-n}dr  \int_\mathbb{S} |f(r\zeta))|^p  d\sigma(\zeta)
\\&\ge c(\alpha)\int_{r_1}^1 r^{n-1} \Phi_n^\alpha(r)(1-r^2)^{-n} dr\int_{\mathbb{S}}|f(r\zeta)|^p{d\sigma(\zeta)}{}\\&\ge c(\alpha) \int_{r_1}^1 r^{n-1} \Phi_n^\alpha(r) (1-r^2)^{-n} dr
(\|f\|^p_{p}-\epsilon).
\end{split}\]
By letting $\alpha\to 1^+$, and noting  that \begin{equation}\label{1m1}1>c(\alpha) \int_{r_1}^1 r^{n-1} \Phi_n^\alpha(r) (1-r^2)^{-n} dr\to 1,\end{equation} when $\alpha\to 1^+$,   we get $$\liminf_{\alpha\to 1^+}\|f\|^p_{\alpha, p}\ge \|f\|^p_{p}-\epsilon$$ and this finishes the proof of \eqref{betap}.

By using the same method as in the first part of the proof precisely \eqref{david} with $p\alpha$ instead of $\alpha$, by using this time that $f$ is bounded we get, we have $$\|f\|_{\alpha,\alpha p}\le \|f\|_{\alpha p}.$$  Then by using the Lesbegue dominated theorem we obtain  \begin{equation}\label{thist}\limsup_{\alpha\to 1^+}\|f\|_{\alpha,\alpha p}\le \|f\|_p.\end{equation}


The opposite inequality is much easier and follows from Jensen's inequality. Indeed $$\int_{\mathbb{S}}|f(r\zeta)|^{\alpha p}d\sigma(\zeta)\ge \left(\int_{\mathbb{S}}|f(r\zeta)|^{ p}d\sigma(\zeta)\right)^\alpha.$$
Hence for $r_1$ satisfying \eqref{hardep} we have
\[\begin{split}\|f\|^{p\alpha}_{\alpha,\alpha p}&=c(\alpha)  \int_0^1 r^{n-1} \Phi_n^\alpha(r)(1-r^2)^{-n}dr  \int_\mathbb{S} |f(r\zeta)|^{p\alpha}  d\sigma(\zeta)
\\&\ge c(\alpha)\int_{r_1}^1 r^{n-1} \Phi_n^\alpha(r)(1-r^2)^{-n} dr\left(\int_{\mathbb{S}}|f(r\zeta)|^p{d\sigma(\zeta)}\right)^\alpha.\end{split}\]
By \eqref{hardep} and \eqref{1m1} we obtain $$\liminf_{\alpha\to 1+}\|f\|^{p\alpha}_{\alpha,\alpha p}\ge \left(\|f\|^p_{p}-\epsilon\right)^\alpha.$$
This finishes the proof of the proposition.

\end{proof}

\begin{proposition}\label{propozicija1} Let $p>0$ and let $f\in\mathbf{h}^p$. Then
\begin{equation}\label{nisapod}\lim_{\alpha\to 1+}\|f\|_{\alpha, \alpha p}=\|f\|_{p}.\end{equation}
\end{proposition}
\begin{proof}[Proof of Proposition~\ref{nisapod}]
We repeat the previous proof. To get \eqref{thist}, we use Theorem~\ref{shtune}, more specifically \eqref{two232}.
\end{proof}


\subsection*{Acknowledgments} I would like to thank the anonymous referee for very helpful comments that had a significant impact on this paper. His  idea is used for proving Proposition~\ref{may2023}. He also draw my attention to the recent related paper \cite{cerde}, where isoperimetric inequality and subharmonicity are applied in the more general context. I would like to thank  Petar  Melentijevi\'c for  drawing my attention to the papers  \cite{kulik,3}. I am thankful to  Kristian Seip and  Aleksei Kulikov for the very helpful discussion.

\end{document}